\documentclass[]{interact}

\usepackage{epstopdf}
\usepackage[caption=false]{subfig}
\usepackage{xcolor}

\usepackage[numbers,sort&compress]{natbib}
\bibpunct[, ]{[}{]}{,}{n}{,}{,}
\makeatletter
\def\NAT@def@citea{\def\@citea{\NAT@separator}}
\makeatother

\theoremstyle{plain}
\newtheorem{theorem}{Theorem}[section]
\newtheorem{lemma}[theorem]{Lemma}
\newtheorem{corollary}[theorem]{Corollary}
\newtheorem{proposition}[theorem]{Proposition}
\theoremstyle{definition}
\newtheorem{definition}[theorem]{Definition}
\newtheorem{example}[theorem]{Example}

\theoremstyle{remark}
\newtheorem{remark}{Remark}

\newtheorem{asumption}[theorem]{Assumption}

\begin{document}


\title{A Set-Valued Lagrange Theorem based on  a Process for Convex  Vector Programming}

\author{
\name{Fernando Garc\'ia-Casta\~no\textsuperscript{a}\thanks{CONTACT Fernando Garc\'ia-Casta\~no  Email: Fernando.gc@ua.es} and M.A. Melguizo Padial}
\affil{Department of Applied Mathematics, University of Alicante, Alicante, Spain; \textsuperscript{a}Orcid: 0000-0002-8352-8235}
}

\maketitle

\begin{abstract}
In this paper we present a new set-valued Lagrange multiplier theorem for constrained convex set-valued optimization problems. We introduce the novel concept of Lagrange process. This concept is a natural extension of the classical concept of Lagrange multiplier where the conventional notion of linear continuous operator is replaced by the concept of closed convex process, its set-valued analogue. The behaviour of this new Lagrange multiplier based on a process is shown to be particularly appropriate for some types of proper minimal points and, in general, when it has a bounded base.
\end{abstract}


\begin{amscode}
 90C29, 90C25, 90C31, 90C48
\end{amscode}

\section{Introduction and Main Results} 
The Lagrange multipliers theorem is the most classic theoretical instrument, and the first one from the historical point of view, to solve constrained optimization problems. 
In fact, Lagrange multipliers play a crucial role in the study of constrained optimization as they provide a natural connection between constrained and their corresponding unconstrained optimization problems. Originally, Lagrange formulated his rule for the optimization of a real-valued function under equality restrictions. 
Since then, numerous research works have generalized this approach. In particular, Lagrange multipliers have  been used in vector optimization  by different authors and from different points of view. For example, in \cite{GoJa99}, G\"otz and Jahn extended the Lagrange multiplier rule to set-valued constrained optimization problems  using the contingent epiderivative as differentiability notion. Later, in \cite{ZheKu06}, Zheng and Ng provided generalized Lagrange multiplier rules as necessary optimality conditions in constrained multiobjective optimization problems making use of coderivatives and normal cones. Coderivatives are also used in \cite{Huang2008} (together with Clarke's normal cones) in order to establish Lagrange multiplier rules for super efficiency optimality conditions of constrained multiobjective optimization problems. The bibliography on the matter is wide and we refer the reader to \cite{WanJey00,BalJim96,AmaTaa97,White1985,Frenk2007,Durea2010,Huang2001,Huy2012,GarciaMelguizo2015,GarciaCastano2015} for a sample of papers related to the subject. In this work, we investigate an abstract convex vector optimization problem with inequality constraints. For such a problem, we formulate a new set-valued Lagrange multiplier rule where the role traditionally played  by linear continuous operators is  played now by closed convex processes (their natural set-valued analogues, see \cite[Chapter~2]{aubin1990set}).  The novel concept of  Lagrange process introduced here is a natural extension of the classical concept of Lagrange multiplier and seems to be specially fitting for vector programming. Let us note that in vector programming the set of optimal points arises naturally as a non-singleton set and  the use of processes, set-valued maps after all, makes it easier to deal with. 
However, the idea of using set-valued maps as Lagrangian multipliers is not entirely new. For example, in  \cite{Tanino1980}  the authors defined a set-valued Lagrangian multiplier for a vector problem as the negative conjugate of a perturbation set-valued map. On the other hand, the use of set-valued maps as dual variables has also been carried out succefully with interesting results. For example, in \cite{Hamel2009,Hamel2011} a new concept of Legendre-Fenchel conjugate is developed giving rise to a new Fenchel-Rockafellar duality type. Some additional results following this enquiry line can be found in \cite{Hamel2014} and \cite[Chapter 4]{Bot2009}.



Next, we state a theorem which summarizes the main results obtained in the paper. It contains the above mentioned Lagrange multiplier rule based on processes and provides three optimality criteria. The undefined notions in the statement are the known standards commonly used in the related literature. In any case, the reader will find the corresponding definitions in the subsequent sections.


\begin{theorem}\label{Main_Theorem} 
Let $X$, $Y$, and $Z$ be normed spaces such that $Y$ and $Z$ are ordered  by  the corresponding cones $Y_+$ and $Z_+$, both having non-empty interior. Take a convex set $\Omega \subset X$,  maps $f:\Omega \rightarrow Y$ and $g:\Omega\rightarrow Z$ such that $f$ is $Y_+$-convex and  $g$ is $Z_+$-convex. 
Assume the existence of $x_1 \in \Omega$ such that $-g(x_1)$ belongs to the interior of $Z_+$.
Then for every $y_0=f(x_0)\in f(\Omega)$ such that $g(x_0)\leq 0$,  there exists a closed convex process  $L_{y_0}:Z\rightrightarrows Y$ such that if $y_0$ is a minimal point of the program
\[
\text{  Min \,} f(x) \text{ \  such that \ } x\in \Omega, \text{ \ } g(x) \leq 0, \text{ \ \ \ \ \ \ } (P(0))
\]
then $y_0$ is a weak minimal point of the (set-valued) program
\[
\text{  Min } f(x) + L_{y_0}(g(x)) \text{ such that } x\in \Omega. \hspace{3cm } (D_{L_{y_0}})  
\]
In addition, we have the following
\begin{itemize}
\item[(i)]  $y_0$ is a minimal point of $(P(0))$ if and only if $y_0$ is a minimal point of $(D_{L_{y_0}})$, provided either $Y_+ \setminus \left\lbrace 0 \right\rbrace $ is open or the cone Graph($L_{y_0}$) has a bounded base,
\item[(ii)] $y_0$ is a positive (resp. Hening global, Hening) minimal point of $(P(0))$ if and only if $y_0$ is a positive (resp. Hening global, Hening) minimal point of  $(D_{L_{y_0}})$,  provided $Y_+$ is pointed,
\item[(iii)]  $y_0$ is a super efficient point of $(P(0))$ if and only if $y_0$ is a super efficient minimal point of  $(D_{L_{y_0}})$, provided $Y_+$ has a bounded base.
\end{itemize}
\end{theorem}
In order to make the paper easier to read, we have split Theorem \ref{Main_Theorem}  into several results stated and proved separately along the text. Namely, the first part  in Theorem~\ref{Main_Theorem} corresponds to  Theorem \ref{TeoPrinc} (i) --stated in Section \ref{Sec_Nec_Opt_Condition}--, Theorem \ref{Main_Theorem} (i) corresponds to Corollary \ref{CorPrinc}  and Theorem \ref{TNorm} (resp. for the first and the second condition at the very end in statement (i)) --stated resp. in Sections \ref{Sec_Nec_Opt_Condition} and \ref{Sec_Suf_Opt_Cond}-- , and finally Theorem \ref{Main_Theorem} ~(ii) and  (iii) correspond to Theorem \ref{ThPrEf} --stated in Section \ref{Sec_Suf_Opt_Cond}--.

The paper is organized as follows. 
Section \ref{Preliminaries and Notations} is devoted to recall some basic definitions and facts dealing with set-valued
maps and ordered vector spaces. 
Section \ref{The Set-valued Lagrange Process} is dedicated to the construction of the new Lagrange multiplier. In such a section, and after some technical results, Definition~\ref{LagrProc} introduces the novel concept of Lagrange process ($L_{y_0}$ in Theorem \ref{Main_Theorem}), and associated to it, Definition \ref{ProgramDL} sets its corresponding unconstrained optimization problem (($D_{L_{y_0}}$) in Theorem \ref{Main_Theorem}). 
After that, Definition~\ref{DefMultLag} introduces the notion of (set-valued) Lagrange multiplier. 
Section \ref{Sec_Nec_Opt_Condition} is devoted to derive some optimality conditions based on a generalized set-valued multiplier rule. Theorem \ref{TeoPrinc} (i) provides the necessary optimality condition claimed in the first part of Theorem~\ref{Main_Theorem}. Besides, Theorem~\ref{TeoPrinc} (ii) and Corollary \ref{CorPrinc} provide two sufficient optimality conditions.
Section \ref{Sec_Suf_Opt_Cond}  focuses on sufficient optimality conditions.  We discuss two approaches. The first one is Theorem \ref{TNorm}, which is based on geometrical aspects of the Lagrange process (it corresponds to Theorem \ref{Main_Theorem} (i)). The other approach is  Theorem \ref{ThPrEf}, which is specific for some types of proper efficiency (it corresponds to Theorem~\ref{Main_Theorem} (ii), (iii)). 
Finally, in Section \ref{Section_Case_Set_Valued} we adapt the results obtained in sections \ref{The Set-valued Lagrange Process}, \ref{Sec_Nec_Opt_Condition}, and \ref{Sec_Suf_Opt_Cond}  to a set-valued optimization problem. As a consequence, we increase the range of applicability of the approach made in this paper. This adaptation is straightforward, reason for which we omit the corresponding proofs.




\section{Preliminaries and Notations}\label{Preliminaries and Notations}

We begin this section recalling some basic definitions and facts dealing with set-valued
maps and ordered vector spaces that we will use throughout the paper. Let  ${{\rm I}\hskip-.13em{\rm R}}_+$ be the set of nonnegative real numbers and $Y$ a normed space. A nonempty convex subset $K\subset Y$ is called a cone if $\alpha K \subset K$, for all $\alpha \in {{\rm I}\hskip-.13em{\rm R}}_+$. Let $Z$ be another normed space, a set-valued map $L: Z \rightrightarrows Y$ is characterized by its graph, which is defined by Graph$(L):=\{(z,y)\in Z\times Y \colon y\in L(z)\}$.  According to \cite[Definition 2.1.1]{aubin1990set},  a set-valued map $L: Z \rightrightarrows Y$ is called a process (resp. linear process) if Graph$(L)$ is a cone (resp. a vector subspace) and it is said to be convex (resp. closed) if Graph$(L)$ is a convex (resp. closed) set. Let $Y_+\subset Y$ be a cone, the order on $Y$ given by $Y_+$ is defined as $y_1\leq y_2$ if and only if $y_2-y_1\in Y_+$  for all $ y_1$, $y_2 \in Y$; in such a case $Y$ is said to be an ordered normed space and $Y_+$ is called the order cone on $Y$. In this context a set-valued map $F: Z \rightrightarrows Y$ is said to be $Y_+$-convex if its epigraph, epi$(F):=\left\lbrace (z,y)\in Z \times Y : y \in F(z)+Y_+  \right\rbrace $, is convex. This definition also applies for conventional point-to-point maps just taking single-valued maps. We denote by Int~$Y_+$ the interior of $Y_+$ in $Y$. We write  $y_1< y_2$ if and only if $y_2-y_1\in$ Int~$Y_+$ for all $ y_1$, $y_2 \in Y$.
Let us fix now a set $A \subset Y$ and some $a\in A$. It is said that $a$ is a minimal point of $A$, written $a\in$ Min$(A)$, if $A  \cap (a-Y_+) \subset a+Y_+$. It is immediate to check that if $Y_+$ is pointed (i.e., $Y_+\cap (-Y_+)=\{0\}$), then $a\in$ Min$(A)$ if and only if $A  \cap (a-Y_+) =\{a\}$. In case Int $Y_+\not = \emptyset$, we say that  $a\in A$ is a weak minimal point of  $A$, written $a\in$ WMin$(A)$, if $A  \cap (a- \mbox{Int}\,Y_+) = \emptyset$. It is clear that Min$(A)\subset$ WMin$(A)$ and, if $Y_+ \setminus \{ 0 \}$ is open, then WMin$(A)=$ Min$(A)$.

Here and subsequently we consider the following. Normed spaces $X$, $Y$, and $Z$ such that $Y$ and $Z$ are ordered by the corresponding order cones $Y_+$ and $Z_+$. We assume that both cones have non-empty interior.  
We fix a convex set $\Omega \subset X$ and two maps $f:\Omega \rightarrow Y$ and $g:\Omega  \rightarrow Z$ such that $f$ is $Y_+$-convex (i.e., $\lambda f(u)+(1-\lambda)f(v)-f(\lambda u+(1-\lambda)v) \in Y_+$, $\forall \lambda \in [0,1]$, $u$, $v \in \Omega$) and  $g$ is $Z_+$-convex (i.e., $\lambda g(u)+(1-\lambda)g(v)-g(\lambda u+(1-\lambda)v) \in Z_+$, $\forall \lambda \in [0,1]$, $u$, $v \in \Omega$). Now, we consider the program
\[
\text{  Min \,} f(x) \text{ \  such that \ } x\in \Omega, \text{ \ } g(x) \leq z, \hspace{2cm}(P(z))
\]
for every $z \in Z$. In this framework, fixed $z\in Z$, we say that 
$x_z \in \left\lbrace x \in \Omega: g(x)\leq z \right\rbrace  $ 
is  a minimal solution of $(P(z))$ if $f(x_z)\in$ Min$(\left\lbrace f(x): x \in \Omega, \, g(x)\leq z \right\rbrace)$, in this case we say that $f(x_z)$ is a minimal point of $(P(z))$. Similarly, 
$x_z \in \left\lbrace x \in \Omega: g(x)\leq z \right\rbrace  $
is said to be a weak minimal solution of $(P(z))$ if
$f(x_z)\in$ WMin$(\left\lbrace f(x): x \in \Omega, \, g(x)\leq z \right\rbrace)$, in this case $f(x_z)$ is said to be a weak minimal point of $(P(z))$. 
Associated to the program $(P(z))$ we consider the following set-valued maps: $\Lambda : Z\rightrightarrows X$ defined by $\Lambda (z):=\{ x \colon  x \in \Omega,\ g(x) \leq z \}$ --the feasible set--, $ W: Z\rightrightarrows Y$ defined by $W (z):= f(\Lambda(z))$ --the image under $f$ of the feasible set--, and $ W_{Y_+}: Z\rightrightarrows Y$ defined by $W_{Y_+} (z):= W (z)  + Y_+ $ --the corresponding ``upper image'' of the feasible set--.

\section{The Set-valued Lagrange Process}\label{The Set-valued Lagrange Process}
The objective in this section is to introduce the novel  set-valued concept of Lagrange process ($L_{y_0}$ in Theorem~\ref{Main_Theorem}), the cornerstone of the paper.  As shown in Theorem~\ref{Main_Theorem}, this new Lagrange multiplier is a natural set-valued extension of the conventional concept of Lagrange multiplier for scalar convex programming. Before defining it, we will provide some necessary technical results. Here and subsequently we will use the following notation. Fixed $(z_0,y_0)\in Z\times Y$ and $\epsilon>0$, we write $B((z_0,y_0),\epsilon)=\{(z,y)\in Z \times Y \colon \parallel z-z_0\parallel_Z+\parallel y-y_0 \parallel_Y<\epsilon\}$, being $\parallel \cdot \parallel_Z$ the norm on $Z$ and  $\parallel \cdot \parallel_Y$ the norm on $Y$.

\begin{lemma}\label{LWConv} 
Let us fix $(z,y) \in Z\times Y$ and $\epsilon>0$. The following statements hold. 
\begin{itemize}
\item[(i)] If $B((z,y),\epsilon)\subset $ Graph $(W_{Y_+})$, then $B((z^*,y^*),\epsilon)\subset$ Graph$(W_{Y_+})$, $\forall (z^*,y^*)\in (z,y)+Z_+\times Y_+$. In particular, $(z,y)\in \mbox{Graph}(W_{Y_+})$ implies $(z,y)+Z_+\times Y_+ \subset \mbox{Graph}(W_{Y_+})$.
\item[(ii)] $\mbox{Graph}(W_{Y_+}) \subseteq Z \times Y$  is a convex set.
\end{itemize}
\end{lemma}
\begin{proof}
It is easily seen that
\[
Graph(W_{Y_+})=\left\lbrace (g(x)+z_+,f(x)+y_+): x \in \Omega, y_+ \in Y_+, z_+ \in Z_+  \right\rbrace.
\]
Now, statement (i) is straightforward and so we omit the proof.  Regarding (ii), \cite[Lemma 8.1.15]{Zalinescu2015} yields  directly the convexity of $Graph(W_{Y_+})$. 
\end{proof}
Let us recall (from the former section) that $\Lambda (0)=\{x \in \Omega \colon  g(x) \leq 0 \}$ and $W (0)= f(\Lambda(0))$.

\begin{definition}
For every $y_0 \in W (0)$,  we define
\begin{equation}\label{defi_S_Y}
\begin{array}{lll}
\mathcal{S}_{Y_+}(y_0) & := & \left\lbrace   h \in (Z \times Y)^* \setminus \lbrace (0,0)\rbrace : h(z', y') \leq h(0, y_0) \leq h(z,y),\right.  \\
           &  & \left. \forall (z', y') \in (-Z_+)\times (y_0-Y_+), \forall (z,y)\in Graph(W_{Y_+}) \right\rbrace.
\end{array}
\end{equation}
We say that every $h\in \mathcal{S}_{Y_+}(y_0)$ separates the sets $Graph(W_{Y_+})$ and  $(-Z_+)\times (y_0-Y_+)$ at $(0,y_0)\in \{0\}\times W(0) \subset Z\times Y$. Note that the set $\mathcal{S}_{Y_+}(y_0)$ may be empty.
\end{definition}

\noindent The following lemma introduces several useful properties of the set of  separating hyperplanes just defined above. Such a lemma may be summarized by saying the following. First, $\mathcal{S}_{Y_+}(y_0)\not = \emptyset$ provided  $y_0$ is a minimal point of $(P(0))$. Furthermore, in case $\mathcal{S}_{Y_+}(y_0)\not = \emptyset$, it contains no ``vertical" hyperplanes under the Slater constraint qualification and it coincides with the set of all the supporting hyperplanes of Graph$(W_{Y_+})$ at  $(0,y_0)$.

\begin{lemma}\label{LSeparatingHyperplane} 
Fix $y_0 \in W (0)$, the following statements hold true.
\begin{itemize}
\item[(i)] If $y_0$ is a minimal point of $(P(0))$, then the set $\mathcal{S}_{Y_+}(y_0)$ is not empty.
\item[(ii)] For every $h \in \mathcal{S}_{Y_+}(y_0)$, there exist $z^*_h \in Z^*$ and  $y^*_h \in Y^*$ such that $h(z,y)=z^*_h(z)+y^*_h(y)$, $\forall (z,y) \in Z\times Y$, and furthermore
\begin{eqnarray}
\langle z^*_h , z_+ \rangle =h(z_+,0)\geq 0,\,\forall z_+\in Z_+;\label{Primera_propiedad_h}\\
\langle y^*_h , y_+ \rangle =h(0,y_+)\geq 0,\, \forall y_+\in Y_+. \label{Segunda_propiedad_h}
\end{eqnarray}
\item[(iii)] If there exists $x_1 \in \Omega$ such that $g(x_1)<0$, then $y^*_h \not = 0$ for all $h \in \mathcal{S}_{Y_+}(y_0)$. \label{LSlater} 

\item[(iv)] Suppose $\mathcal{S}_{Y_+}(y_0) \not=\emptyset$. Then we have 
\begin{gather*}
\mathcal{S}_{Y_+}(y_0) =\left\lbrace   h \in (Z \times Y)^* \setminus \lbrace (0,0) \rbrace : h(0, y_0) \leq h(z,y),\, \forall (z,y)\in Graph(W_{Y_+}) \right\rbrace 
\end{gather*}

\end{itemize}
\end{lemma}
\begin{proof}
(i) Since $y_0$ is a minimal point of $(P(0))$, $(0,y_0)\in$ Graph$(W_{Y_+})$. Applying Lemma \ref{LWConv} (i), we have $Z_+\times(y_0+Y_+) \subset$ Graph$(W_{Y_+})$. Since such cones have non-empty interior, Int(Graph$(W_{Y_+})$)$\not = \emptyset$. Furthermore,  $(-Z_+)\times (y_0-Y_+)$ contains no interior points of the set Graph$(W_{Y_+})$. On the contrary, suppose that $(-z_+,y_0-y_+)\in$ Int(Graph$(W_{Y_+})$) for some $(z_+,y_+)\in Z_+\times Y_+$. By Lemma \ref{LWConv} (i), $(0,y_0)\in$ Int(Graph$(W_{Y_+})$), contrary to the minimality of $y_0$. Indeed, if there exists some $\epsilon>0$ such that $B((0,y_0),\epsilon)\subset \mbox{Graph}(W_{Y_+})$, we can choose some $u \in Y_+$ in the unit sphere of $Y$ such that $(0,y_0-\frac{\epsilon}{2}u) \in \mbox{Graph}(W_{Y_+})$. Hence, there exists $x\in \Lambda (0)$ such that $y_0-\frac{\epsilon}{2}u \geq f(x)$, which implies $f(x)<y_0$, a contradiction. Now, Eidelheit separation theorem \cite[Theorem 3 p. 133]{Luenberger1969} provides a hyperplane which  separates the sets Graph$(W_{Y_+})$  and $(-Z_+)\times (y_0-Y_+)$.

(ii)  Fix $h \in \mathcal{S}_{Y_+}(y_0)$. Clearly we can pick $z^*_h \in Z^*$,  $y^*_h \in Y^*$ such that $h(z,y)=z^*_h(z)+y^*_h(y)$, $\forall (z,y) \in Z\times Y$. Thus $\langle z^*_h , z_+  \rangle  =h(z_+,0)$ and $\langle y^*_h , y_+  \rangle  =h(0,y_+)$ for all $ z_+\in Z_+$ and $ y_+\in Y_+$. By Lemma \ref{LWConv} (i), $(z_+,y_0) \in \text{Graph} (W_{Y_+})$ for all $z_+\in Z_+$,  and then $h(0,y_0) \leq h(z_+,y_0)$. Consequently, $h(z_+,0)= h(z_+,y_0) - h(0,y_0)\geq 0$ which gives (\ref{Primera_propiedad_h}). The proof for (\ref{Segunda_propiedad_h}) is similar.


(iii) Fix $h\in \mathcal{S}_{Y_+}(y_0)$ and suppose $y^*_h = 0$.  Since $h \not =0$, we have $ z^*_h \not = 0$. Therefore $z^*_h(z)>0$ for every $z \in$ Int$\,Z_+$. On the other hand, since $(g(x_1),f(x_1)) \in \mbox{Graph}( W_{Y_+})$, we have  $z^*_h(0)+y^*_h(y_0)=h(0,y_0)\leq h(g(x_1),f(x_1))=z^*_h(g(x_1))+y^*_h(f(x_1))$. Thus $0 \leq z^*_h(g(x_1))$, a contradiction because $g(x_1) \in$ $-$Int$\,Z_+$.    

(iv) The inclusion ``$\subseteq$'' is immediate. Let us check ``$\supseteq$''. Fix an arbitrary $\bar{h}  \in (Z \times Y)^* \setminus \lbrace (0,0)\rbrace$ such that $\bar{h}(0, y_0) \leq \bar{h}(z,y)$ for all $(z,y)\in \mbox{Graph}(W_{Y_+})$. We will check that $\bar{h}(-z_+,y_0-y_+)\leq \bar{h}(0,y_0)$ for any $z_+ \in Z_+ $ and $y_+ \in Y_+$. 
Indeed, since $\bar{h}(z_+,y_0+y_+)\in \mbox{Graph}(W_{Y_+})$ because of Lemma  \ref{LWConv} (i), then $\bar{h}(0, y_0) \leq \bar{h}(z_+,y_0+y_+)$ for any $z_+ \in Z_+ $ and $y_+ \in Y_+$. Therefore $0 \leq \bar{h}(z_+,y_+)$, and consequently $0 \geq\bar{h}(-z_+,-y_+)$, from which it follows that $\bar{h}(0,y_0)\geq \bar{h}(-z_+,y_0-y_+) $.

\end{proof}

\noindent In order to guarantee that $y^*_h \not = 0 $ for all $h \in \mathcal{S}_{Y_+}(y_0)$, from now on we assume  the Slater constraint qualification as true.

\begin{asumption}\label{Hip_Slater}[Slater constraint qualification]
From now on we assume that there exists $x_1 \in\Omega$ such that $g(x_1)<0$.
\end{asumption}

\begin{definition}  \label{DL} 
Let us fix $y_0 \in W(0)$. We define, for every $h \in  \mathcal{S}_{Y_+}(y_0)$, the following two processes
\begin{itemize}
\item[] $\bar{L}_h:Z\rightrightarrows Y, \ \bar{L}_h(z):= \left\lbrace  y\in Y: h (z,-y)=0\right\rbrace$,  for all $ z \in Z$,
\item[] $ L_h:Z\rightrightarrows Y, \ L_h(z):= \bar{L}_h(z) + Y_+$, for all $ z \in Z$.
\end{itemize}
\end{definition}
The next lemma provides three necessary technical properties in order to introduce, subsequently, the notion of Lagrange process properly.
\begin{lemma} \label{LGrafLalfa} 
Fix $y_0 \in W(0)$ and suppose  $\mathcal{S}_{Y_+}(y_0) \not= \emptyset $. The following statements hold.
\begin{itemize}
\item[(i)] Graph$(L_h)=\left\lbrace (z,y) \in Z \times Y : h(z,-y)\leq 0 \right\rbrace$, for all  $ h \in S_{Y_+}(y_0)$.
\item[(ii)] If $(\bar{z},\bar{y})\in \mbox{Graph}(W_{Y_+})$, then $(-\bar{z},\bar{y}-y_0)\in \mbox{Graph}(L_h)$, for all  $ h \in S_{Y_+}(y_0)$.
\item[(iii)] $\bigcap_{h \in  \mathcal{S}_{Y_+}(y_0)}{Graph(L_h)}\not = \{ (0,0) \}$. 
\end{itemize}

\end{lemma}
\begin{proof}
(i) Let us prove the equality by double inclusion. If $(z,y)\in$ Graph$(L_h)$, then $y=\bar{y}+y_+$ with $h(z,-\bar{y})=0$ and $y_+ \in Y_+$. Hence, by Lemma \ref{LSeparatingHyperplane} (ii) $h(z,-y)=h(z,-\bar{y})+h(0,-y_+)\leq 0$. Now, let us check the reverse inclusion. Fix $(z_0,y_0) \in Z \times Y$ such that $h(z_0,-y_0) \leq 0$. We define $h'(z,y):= h(z,-y)$ for every $(z,y) \in Z \times Y $. We claim that $\{ 0 \} \times \text{Int} \, Y_+\nsubseteq \text{Ker } h'$. Otherwise, $y^*_h = 0 $, which contradicts Lemma~\ref{LSlater}~(iii). Let us pick some $(0,y_+) \in \{0\}\times \text{Int } Y_+ \setminus \text{Ker } h'$. We decompose $(z_0,y_0)=(z_0, \bar{y}_0) + \lambda (0,y_+)$, for some $\lambda \in \mathbb{R}$ and $(z_0, \bar{y}_0) \in \text{Ker } h'$. Then, on the one hand, $h(z_0,-\bar{y}_0)=h'(z_0,\bar{y}_0)=0$, or equivalently, $\bar{y}_0\in \bar{L}_h(z_0)$. On the other hand, we have
$ 0\geq h'(z_0,y_0)= h'(\bar{z}_0, \bar{y}_0) + \lambda h'(0, y_+)=\lambda  h(0, -y_+)$. Since
$h(0, -y_+)\leq 0 $ by Lemma \ref{LSeparatingHyperplane} (ii), we have $\lambda \geq 0 $. Then $y_0=\bar{y}_0+\lambda y_+ \in \bar{L}_h(z_0)+Y_+=L_h(z_0)$.

(ii) Fix $(\bar{z}, \bar{y}) \in \mbox{Graph}(W_{Y_+})$ and an arbitrary $h \in S_{Y_+}(y_0)$. Since $h(0,y_0)\leq h(\bar{z}, \bar{y})$, $h(\bar{z}, \bar{y}-y_0)\geq 0 $. Now statement (i) yields $(-\bar{z}, \bar{y}-y_0) \in Graph(L_h)$.

(iii) Direct consequence of statement (ii). 
\end{proof}

\begin{definition} \label{LagrProc}
Fix $y_0\in W(0)$. We define the Lagrange process of $(P(0))$ at $y_0$ as
the closed convex process  $L_{y_0}:Z\rightrightarrows Y $  such that
\begin{equation*}
Graph(L_{y_0}):=\left\lbrace 
\begin{array}{lll}
\displaystyle \bigcap_{h \in S_{Y_+}(y_0)}{Graph(L_h)}, & \text{ if } & \mathcal{S}_{Y_+}(y_0) \not= \emptyset, \\
& & \\
\displaystyle Z \times Y, & \text{ if } & \mathcal{S}_{Y_+}(y_0) = \emptyset.
\end{array}
\right. 
\end{equation*}
\end{definition} 
\begin{remark}
Let us observe that the Lagrange process $L_{y_0}$ set in Definition \ref{LagrProc}  is a closed convex process, since  Graph$(L_{y_0})$ is the intersection of closed halfspaces according to Lemma \ref{LGrafLalfa} (i).
\end{remark}
\begin{definition}\label{ProgramDL} 
Given $L:Z\rightrightarrows Y$ a closed convex process, we define the program 
\[
\text{  Min } f(x) + L(g(x)) \text{ such that } x\in \Omega. \hspace{2cm } (D_L)  
\]
Besides, we say that $x_0\in \Omega$  is:
\begin{itemize}
\item[(i)] A minimal solution of $(D_L)$ if $f(x_0)\in$ Min$(\left\lbrace f(x) + L(g(x)): x \in \Omega \right\rbrace)$, in this case we say that $f(x_0)$ is a minimal point of $(D_L)$. 
\item[(ii)] A weak minimal solution of $(D_L)$ if
$f(x_0)\in$ WMin$(\left\lbrace f(x) + L(g(x)): x \in \Omega \right\rbrace)$, in this case we say that $f(x_0)$ is a weak minimal point of $(D_L)$. 
\end{itemize}
\end{definition}

\begin{definition}\label{DefMultLag} 
Fix $y_0\in W(0)$. We say that  a closed convex process $L:Z\rightrightarrows Y$ is a  Lagrange
multiplier (resp. weak Lagrange multiplier) of ($P(0)$) at  $y_0$ if $y_0$ is a minimal (resp. weak minimal) point of the corresponding program ($D_L$).
\end{definition}

\section{Necessary Optimality Condition}\label{Sec_Nec_Opt_Condition}
In this section we turn back to Theorem \ref{Main_Theorem} in the introduction. Our aim here is to state and prove the first statement of such a result and the first part of the corresponding assertion (i).  The above mentioned parts of Theorem \ref{Main_Theorem} corresponds --as was said in the introduction--  to Theorem \ref{TeoPrinc} (i) and Corollary  \ref{CorPrinc} in this section, respectively. On one hand, Theorem~\ref{TeoPrinc}~(i) provides a  necessary  optimality condition for the minimal points of program ($P(0)$) which, unfortunately, is  not sufficient (as Example \ref{Example} shows later). In order to relieve this gap, Theorem \ref{TeoPrinc} (ii) gives a first  sufficient condition of optimality that supplies  some additional information about the relationship between the minimal points of programs ($P(0)$) and $(D_{L_{y_0}}$). Later, we get Corollary \ref{CorPrinc} which provides another sufficient condition of optimality and shows that the Lagrange process $L_{y_0}$ (introduced in Definition \ref{LagrProc}) is a natural set-valued extension of the conventional concept for Lagrange multiplier for scalar convex programming. Besides, this corollary will turn out very fruitful later.

The first result in this section is a technical lemma which will be used later.
\begin{lemma}\label{Lemma_previo_Tma_principal}
Fix $y_0\in W(0)$ and let $L_{y_0}$ be the Lagrange process of $(P(0))$ at $y_0$. The following assertions hold true.
\begin{itemize}
\item[(i)] $0 \in L_{y_0}(-z_+)$, for every $z_+ \in Z_+$. 
\item[(ii)] $W(0) \subseteq
\left\lbrace f(x)+ L_{y_0}(g(x)): x \in \Omega  \right\rbrace$.
\item[(iii)] $Graph(W_{Y_+}) \subset (0,y_0) + \{(z,y)\in Z\times Y \colon (-z,y)\in \mbox{Graph}(L_{y_0})\}$.
\end{itemize}
\end{lemma}

\begin{proof}
(i) If $S_{Y_+}(y_0)=\emptyset$ there is nothing to prove. Then, we assume $S_{Y_+}(y_0)\not =\emptyset$ and fix an arbitrary $z_+\in Z_+$. Since $y_0 \in W_{Y_+}(0)$, by Lemma \ref{LWConv} (i), we have that $ y_0 \in W_{Y_+}(z_+)$. Besides, by Lemma \ref{LGrafLalfa} (ii), $(-z_+,0) \in \mbox{Graph}(L_h) $ for all $h \in S_{Y_+}(y_0)$. Consequently $(-z_+,0) \in$ Graph$(L_{y_0})$ by definition of $L_{y_0}$.

(ii) On one hand we have 
\[\left\lbrace f(x)+ L_{y_0}(g(x)): x \in \Omega, \, g(x)\leq 0 \right\rbrace \subseteq
\left\lbrace f(x)+ L_{y_0}(g(x)): x \in \Omega  \right\rbrace.\]
On the other hand, by assertion (i), $ g(x)\leq 0$ implies $0 \in L_{y_0}(g(x))$. Hence,
\[W(0)=\left\lbrace f(x)+ 0: x \in \Omega, \, g(x)\leq 0 \right\rbrace \subseteq
\left\lbrace f(x)+ L_{y_0}(g(x)): x \in \Omega , \, g(x)\leq 0 \right\rbrace.\]

(iii) If $S_{Y_+}(y_0)=\emptyset$ there is nothing to prove. Then, we assume $S_{Y_+}(y_0)\not =\emptyset$ and take $(\bar{z},\bar{y})\in$ Graph$(W_{Y_+})$. We write $(\bar{z},\bar{y})=(0,y_0)+(\bar{z},\bar{y}-y_0)$. By Lemma \ref{LGrafLalfa} (ii), $(-\bar{z},\bar{y}-y_0) \in $ Graph$(L_h)$ for all  $h \in S_{Y_+}(y_0)$. Thus $(-\bar{z},\bar{y}-y_0) \in $ Graph$(L_{y_0})$ by definition of $L_{y_0}$. 
\end{proof}


\begin{proposition}\label{PropLalfa} 
Let us fix $y_0 \in W(0)$ and $h \in S_{Y_+}(y_0)$. Then  $y_0$ is a weak minimal point of the program
\[
\text{  Min \,} f(x) + L_h(g(x)) \text{ \  such that \ } x\in \Omega. \hspace{2cm} (D_{L_h})
\]
\end{proposition}
\begin{proof}
Fix  $h \in S_{Y_+}(y_0)$, and suppose that $y_0$ is not a weak minimal point of ($D_{L_h}$). Then, there exist $\bar{x} \in \Omega$,  $\bar{y} \in \bar{L}_h(g(\bar{x}))$, and $y_+ \in Y_+ $   such that $f(\bar{x})+\bar{y}+y_+ \in y_0 - \text{Int\,} Y_+$.
Define $\hat{y}:=f(\bar{x})+\bar{y}+y_+ - y_0 \in - \text{Int\,} Y_+$.   
On one hand, by Lemma \ref{LSeparatingHyperplane}~(ii), $h(0,\hat{y})=\langle y^*_h, \hat{y} \rangle  \leq 0$. By Assumption \ref{Hip_Slater} and Lemma \ref{LSlater} (iii) we have $y^*_h \not =0$, and then we get
\begin{equation}
h(0,\hat{y})=\langle y^*_h, \hat{y} \rangle  \, <0. \label{des1} 
\end{equation}
On the other hand, $(0, \hat{y})=(g(\bar{x}), f(\bar{x})) - (g(\bar{x}),-\bar{y}) + (0,y_+)-(0,y_0)$. Since $(g(\bar{x}), f(\bar{x})) \in$ Graph$(W_{Y_+})$, $h (g(\bar{x}), f(\bar{x}))\geq h(0,y_0) $. Besides $h(g(\bar{x}), -\bar{y})= 0 $ (because $\bar{y} \in \bar{L}_h(g(\bar{x}))$) and $ h(0,y_+)\geq 0 $ (because $y_+ \in Y_+$). It follows that
$h(0, \hat{y})=h(g(\bar{x}), f(\bar{x})) - h(g(\bar{x}),-y) + h(0,y_+)-h(0,y_0)\geq 0$,
which contradicts (\ref{des1}). 
\end{proof}

\noindent The following theorem constitutes the benchmark of the whole paper. In fact, the rest of the article spins around it.

\begin{theorem}\label{TeoPrinc} 
Fix $y_0\in W(0)$ and let $L_{y_0}$ be the Lagrange process of $(P(0))$ at $y_0$. The following statements hold true. 
\begin{itemize}
\item[(i)] If $y_0$ is a minimal point of $(P(0))$, then $y_0$ is a weak minimal point of $(D_{L_{y_0}})$.
\item[(ii)] If $y_0$ is a minimal (resp. weak minimal) point of $(D_{L_{y_0}})$, then $y_0$ is a minimal (resp. weak minimal) point of $(P(0))$.
\end{itemize}
\end{theorem}

\begin{proof}
$(i)$  Take $h_0 \in S_{Y_+}(y_0)$. From Proposition \ref{PropLalfa}, $y_0$ is a weak minimal point of $(D_{L_{h_0}})$. Hence, there are no $\bar{x} \in \Omega$ and  $y' \in L_{h_0}(g(\bar{x}))$ such that $f(\bar{x})+ y' \in y_0 - \text{Int\,} Y_+$.
Now, since Graph$(L) \subset$ Graph$(L_{h_0})$, there are no
$\bar{x} \in \Omega$ and $y' \in L(g(\bar{x}))$ such that $f(\bar{x})+ y' \in y_0 - \text{Int\,} Y_+$. 

$(ii)$ We only prove the case of minimality, the proof for the case of weak minimality is similar. If $y_0$ is not a minimal point of $(P(0))$, then there exists $\bar{x} \in \Omega$ with $g(\bar{x}) \leq 0 $  such that $f(\bar{x}) \in y_0 -  Y_+$. By Lemma~\ref{Lemma_previo_Tma_principal}~(ii), $f(\bar{x}) =f(x)+y'$ for some $x \in \Omega$ such that $y'\in L_{y_0}(g(x))$. Hence $f(x)+y'\in  y_0 -  Y_+$, and so, $y_0$ is not a minimal point of  $(D_{L_{y_0}})$.
\end{proof}

\begin{remark}
According to Theorem \ref{TeoPrinc} (i), every minimal point of program $(P(0))$ has an associated weak Lagrange multiplier $L_{y_0}$ that coincides with its corresponding Lagrange process. On the other hand, by Theorem \ref{TeoPrinc} (ii), such a process $L_{y_0}$ provides  a sufficient optimality condition through the dual program $(D_{L_{y_0}})$.
\end{remark}
In the following example  we consider a particular family of programs $(P(z))$ from Section \ref{Preliminaries and Notations}. Such an example shows that statement (i) in Theorem \ref{TeoPrinc} is not sufficient.
\begin{example}\label{Example}
Let $X=Y=\mathbb{R}^2$, $Y_+=\left\lbrace (x_1,x_2) \in \mathbb{R}^2 : x_1 \geq 0, x_2 \geq 0  \right\rbrace$, $Z=\mathbb{R}$, $Z_+=\mathbb{R}_+$, $D= \left\lbrace (x_1,x_2) \in \mathbb{R}^2 : x_2 > 0  \right\rbrace \cup \lbrace (0,0) \rbrace$, $f(x_1,x_2)=(x_1,x_2)$, and $g(x_1,x_2)=x_2-1$. Then $y_0=(0,0)$ is a minimal point of ($P(0)$)  but not a minimal point of $(D_{L_{y_0}})$.
\end{example}
\begin{proof}
Indeed, we have 
\begin{equation*} 
\Lambda(z) =
\begin{cases}
\left\lbrace (x_1,x_2) \in \mathbb{R}^2 : 0 < x_2 \leq 1+z  \right\rbrace \cup \lbrace (0,0) \rbrace,& \text{ if } z \in [-1,\infty), \\
\emptyset, & \text{ otherwise.}
\end{cases}
\end{equation*}
Besides, $W(z)=f(\Lambda(z))=\Lambda(z)$ and
\begin{equation*} 
W_{Y_+}(z)= f(\Lambda(z))+Y_+ = \begin{cases}
A, &\text{   if } z \in [-1,\infty), \\
\emptyset,&  \text{ otherwise,} 
\end{cases} 
\end{equation*}
where $A=\lbrace (x_1,x_2) \in \mathbb{R}^2 : 0 \leq x_1, 0 \leq x_2 \rbrace \cup \lbrace (x_1,x_2) \in \mathbb{R}^2 : x_1 < 0 < x_2 \rbrace$. So, it is immediate that
\begin{equation*} 
\mbox{Min}(\left\lbrace f(x): x \in \Omega, \, g(x)\leq z \right\rbrace)= \begin{cases}
\lbrace (0,0) \rbrace,& \text{   if } z \in [-1,\infty), \\
\emptyset, &  \text{ otherwise. } 
\end{cases} 
\end{equation*}

Now, consider the particular program ($P(0)$), and on it, $(0,y_0)=(0,(0,0))$. Now we get that Graph$(W_{Y_+})= [-1, \infty) \times A \subset \mathbb{R}^3$. Thus $S_{Y_+}=\left\lbrace (0,0,\lambda):\lambda >0 \right\rbrace$, which yields  Graph$(L_{y_0})=\mathbb{R} \times \mathbb{R} \times \mathbb{R}_+ $. Consequently $L_{y_0}(z)= \mathbb{R} \times \mathbb{R}_+ $ for all $z \in Z=\mathbb{R} $. It follows that
\[ 
\left\lbrace f(x_1,x_2)+L_{y_0}(g(x_1,x_2)):  (x_1,x_2) \in \Omega \right\rbrace=\Omega + (\mathbb{R} \times \mathbb{R}_+)=\mathbb{R} \times \mathbb{R}_+,
\] 
and hence,
\[
\mbox{Min}( \left\lbrace f(x_1,x_2)+L_{y_0}(g(x_1,x_2)): (x_1,x_2) \in \Omega \right\rbrace  \mid Y_+)=\mbox{Min}(\mathbb{R} \times \mathbb{R}_+ \mid \mathbb{R} \times \mathbb{R}_+)=\emptyset.
\]
Thus, $y_0=(0,0)$ is not a minimal point of $(D_{L_{y_0}})$; however it is a weak minimal point of $(D_{L_{y_0}})$ since 
\[
\left( \mathbb{R} \times \mathbb{R}_+\right)  \cap \left( (0,0) - \text{Int} ( \mathbb{R} \times \mathbb{R}_+) \right) =
\left( \mathbb{R} \times \mathbb{R}_+\right)  \cap \left(  \mathbb{R} \times (\mathbb{R}_{+}\setminus \{0\}) \right) =\emptyset. 
\]
\end{proof}

\noindent The following  corollary is an immediate consequence of Theorem \ref{TeoPrinc}. 

\begin{corollary}\label{CorPrinc} 
Suppose that $Y_+ \setminus \{ 0 \}$ is open. Fix $y_0\in W(0)$ and let $L_{y_0}$ be the Lagrange process of $(P(0))$ at $y_0$. Then, $y_0$ is a minimal point of $(P(0))$ if and only if $y_0$ is a minimal point of $(D_{L_{y_0}})$.
\end{corollary}

\begin{proof}
Since  $Y_+  \setminus \{ 0 \}$ is taken open, the concepts of minimal and weak minimal points coincide. Now, a direct application of Theorem \ref{TeoPrinc} yields the result. 
\end{proof}

\noindent The next remark shows how the traditional Lagrange multiplier theorem for scalar convex programming remains in Theorem \ref{TeoPrinc} as a particular case.
\begin{remark}
The application of Corollary \ref{CorPrinc} to the case $Y=\mathbb{R}$ and $Y=\mathbb{R}_+$ yields the classical Lagrange multiplier theorem for the scalar convex programming. Indeed, the set $Y_+  \setminus \{ 0 \}=(0,+\infty)$ is  open and the concepts of minimal and minimum coincide in the scalar case.
\end{remark}

\section{Sufficient Optimality Conditions}\label{Sec_Suf_Opt_Cond}
The objective in this section is to set up some sufficient optimality criteria for the program $(P(0))$. In particular, we will establish the second part of statement (i) of Theorem \ref{Main_Theorem} (in the introduction)  together with assertions (ii) and (iii). Let us note that the second part of statement (i) corresponds to Theorem \ref{TNorm} next in the first subsection, and statements (ii) and (iii) corresponds to Theorem \ref{ThPrEf}  in the second subsection. In this way, we provide two types of criteria, some based on the geometry of the Lagrange process $L_{y_0}$ (in the next subsection), and other based on the particular characteristics of the minimal points $y_0$ involved (in the second subsection).

\subsection{Geometric Optimality Conditions}
By statement (i) in Theorem \ref{TeoPrinc} any $y_0\in W(0)$ minimal point of $(P(0))$ is also a weak minimal point of $(D_{L_{y_0}})$. Next, we will prove that if the graph of the corresponding Lagrange process $L_{y_0}$ has a bounded base (or equivalently, if its vertex is a denting point), then $y_0$ becomes a minimal point of $(D_{L_{y_0}})$. This result will turn out a useful tool to distinguish the minimal points of program $(P(0))$ from among the weak minimal points of the unconstrained program $(D_{L_{y_0}})$.

For this purpose, we will prove some previous technical results. First, we recall some basic concepts. Fixed a normed space $E$, the negative polar cone to subset $S_1 \subseteq E$ (resp. $S_2 \subseteq E^*$) is defined by $S_1^-:=\{ e^* \in E^*: \langle e^*,s \rangle  \leq 0,\,   \forall  s \in S_1 \}$ (resp. $S_2^-:=\{ e \in E: \langle s^*,e \rangle   \leq 0,\, \forall s^* \in S_2\}$). Their corresponding positive polar cones are defined by $S_1^+:=-S_1^-\subset E^*$ and $S_2^+:=-S_2^-\subset E$. 
Let us recall that a base for a cone $K$ in a normed space $E$ is a nonempty convex subset $B\subset K$ such that $0 \not \in \overline{B}$ and each $k\in K\setminus \{0\}$ has a unique representation of the form $k=\lambda b$ for some $\lambda>0$ and $b\in B$. A base is called a bounded base if it is  bounded as a subset of $E$. Bounded bases have been widely studied (see, for example \cite{GARCIACASTANO20151178} and the references therein) and they provide a useful tool in many topics such as in the theory of Pareto minimal points in \cite{Borwein1993}, in reflexivity of Banach spaces in \cite{Casini2010a}, and in density theorems in \cite{Gong1995}. In the following result, this notion will turn out useful again.

\begin{theorem}\label{TNorm} 
Fix $y_0\in W(0)$ and  let $L_{y_0}$ be the Lagrange process of $(P(0))$ at $y_0$.  Suppose that the cone  Graph$(L_{y_0})$ has a bounded base. Then $y_0$ is a minimal point of $(P(0))$ if and only if $y_0$ is a minimal point of  $(D_{L_{y_0}})$.
\end{theorem}

\begin{proof} 
We will prove the first implication by contradiction. Suppose that $y_0$ is a minimal point of $(P(0))$ but it is not a minimal point of $(D_{L_{y_0}})$. Hence, there exist $\bar{x}\in \Omega$ and $\bar{y} \in L_{y_0}(g(\bar{x}))$ such that
\begin{equation}
f(\bar{x})+\bar{y} \in y_0 -  Y_+ \setminus \ \{ 0 \}. \label{eq2}
\end{equation}
By the definition of $L_{y_0}$, we have $\bar{y} \in L_h (g(\bar{x}))$ for every $h \in S_{Y_+}(y_0)$. Then, for every $h \in S_{Y_+}(y_0)$ we can decompose $\bar{y}=y'_h + y_h^+$ where $y'_h \in \bar{L}_h (g(\bar{x})) $ and $y_h^+ \in Y_+ $.  
Now, we denote $\hat{y}: = f(\bar{x})+\bar{y} - y_0\in -Y_+\setminus \{0\}$, that can be expressed as 
$\hat{y} = f(\bar{x})+y'_h + y_h^+ - y_0$ for every  $h \in S_{Y_+}(y_0)$. Then 
\begin{equation}
h(0, \hat{y})=h(g(\bar{x}), f(\bar{x})) - h(g(\bar{x}),-y'_h) + h(0,y_h^+)-h(0,y_0)\geq 0,
\label{eq3}
\end{equation} for all $h \in \mathcal{S}_{Y_+}(y_0)$ because $(g(\bar{x}), f(\bar{x})) \in \mbox{Graph}(W_{Y_+})$ and $y'_h\in \bar{L}_h(g(\bar{x}))$.

However, we will prove the existence of some $h_0 \in \mathcal{S}_{Y_+}(y_0)$ such that $h_0(0,\hat{y})~<~0$, a contradiction. Indeed,  since  Graph$(L_{y_0})$ has a bounded base, by \cite[Theorem 1.1]{GARCIACASTANO20151178}, Int$\left( \mbox{Graph}(L_{y_0})\right)^-\neq \emptyset$. Now, we pick
\begin{equation}
\bar{h} (\cdot,\cdot)=\langle\bar{z}^*, \cdot  \rangle  + \langle\bar{y}^* , \cdot \rangle  \, \in \, -\text{Int}  \left( \mbox{Graph}(L_{y_0})\right)^-, \label{hbarra}
\end{equation} 
and define the following continuous linear map $h_0(\cdot,\cdot)=\langle-\bar{z}^*, \cdot  \rangle  + \langle\bar{y}^*, \cdot \rangle $. Let us check that $h_0 \in \mathcal{S}_{Y_+}(y_0)$.  Indeed, if $(\bar{z},\bar{y} ) \in$ Graph$(W_{Y_+})$, by Lemma \ref{Lemma_previo_Tma_principal} (iii) we can decompose $(\bar{z},\bar{y} )= (0,y_0)+ (z', y')$ with $ (-z', y') \in$ Graph$(L_{y_0})$. From (\ref{hbarra}) we have $0  <  \bar{h}(-z',y')= \langle\bar{z}^*, -z'  \rangle  + \langle\bar{y}^*, y' \rangle  = \langle-\bar{z}^*, z'  \rangle  + \langle\bar{y}^*, y' \rangle  =h_0(z',y')$. Therefore $h_0(\bar{z},\bar{y})= h_0(0,y_0)+h_0(z',y')> h_0(0,y_0)$. Consequently, by statement (iv) of Lemma~\ref{LSeparatingHyperplane} we obtain $h_0 \in \mathcal{S}_{Y_+}(y_0).$  On the other hand, by Lemma~\ref{LWConv}~(i), $ (0,y_0) +(0,-\hat{y}) \in$ Graph$(W_{Y_+}).$ Thus, by Lemma~\ref{Lemma_previo_Tma_principal} (iii), we have $(0,-\hat{y}) \in$ Graph$(L_{y_0})$. Therefore, $h_0(0,\hat{y})=\langle\bar{y}^*, \hat{y}\rangle=\bar{h}(0,\hat{y})<0$, which contradicts (\ref{eq3}). As a consequence  $y_0$ is a minimal point of $(D_{L_{y_0}})$.

The reverse implication is a direct consequence of statement (ii) in Theorem \ref{TeoPrinc}. 
\end{proof}

\begin{remark}
Example \ref{Example} above also shows that the condition ``bounded'' can not be \break dropped down from the statement of the former theorem.
\end{remark}
By \cite[Theorem 1.1]{GARCIACASTANO20151178}, the assumption in the former theorem is equivalent to the property that the vertex is a denting point of the cone Graph$(L_{y_0})$. Example~\ref{Example} also shows that such a condition cannot be weakened to the property that the origin is a point of continuity for the cone Graph$(L_{y_0})$. For more information about the notion of point of continuity we refer the reader to \cite{GARCIACASTANO2018} and the references therein.

Despite what we have noted in the former paragraph, Theorem \ref{TNorm} can be weakened in the way as the following result shows. From now on, bd$A$ denotes the boundary of the set $A$.

\begin{corollary}\label{CoroTNorm} 
Fix $y_0=W(0)$ and  let $L_{y_0}$ be the Lagrange process of $(P(0))$ at $y_0$.  Suppose that there exists $h_{0} \in \mathcal{S}_{Y_+}(y_0)$  such that $h_{0}(0,y)>0$, for every $y \in \text{bd\,}(Y_+) \setminus \{0\}$. Then $y_0$ is a minimal point of $(P(0))$ if and only if $y_0$ is a minimal point of  $(D_{L_{y_0}})$
\end{corollary}

\begin{proof}
We follow the notation of the proof of Theorem \ref{TNorm} and, same as there, we suppose that $y_0$ is a minimal point of $(P(0))$ but  not a minimal point of $(D_{L_{y_0}})$. We consider again the element  $\hat{y} \in -Y_+ \setminus \{ 0 \}$. In fact, $\hat{y} \in  \text{bd} (-Y_+)$ because $y_0$ is a weak minimal point of $(D_{L_{y_0}})$. Now, the argument used in the proof of Theorem \ref{TNorm} works again by $h_0$ in the statement of this result. 
\end{proof}

\subsection{Sufficient Conditions for some Proper Minimal Points}\label{Sec_Proper_Efficiency}
We  finish this section making use of Corollary \ref{CorPrinc} to prove  Theorem \ref{ThPrEf}. On that, we see how  some types of proper minimal points of program $(P(0))$ become proper minimal points of the program $(D_{L_{y_0}})$. Therefore, restricting our approach to these types of proper minimal points we avoid certain types of anomalous situations and the two assertions in Theorem \ref{TeoPrinc} become equivalences.
 

Next, we will define the four types of proper minimal points we will deal with. Throughout the reminder of this section the order cone $Y_+$ of $Y$ is assumed to be pointed (i.e., $Y_+ \cap (-Y_+)=\{0\}$). By $Y_+^{+i}$ we denote the set of elements $f \in Y^*$ such that $f(y_+)>0$, for all $ y_+ \in Y_+\setminus \{0\}$. Let us recall that $Y_+$ has a base if and only if $Y_+^{+i}\not = \emptyset$ (see, for example \cite{Jahn2004}). We denote by cone($A$) the cone generated by the set $A$, i.e., cone($A$)$:=\{ta\colon t \geq 0,\, a \in A\}$, by $\overline{\mbox{cone}}(A)$ the closure of the cone generated by the set $A$, and by $B_{Y}$ the closed unit ball of $Y$.
\begin{definition} Let $Y$ be an ordered normed space having a pointed order cone. We say that $\bar{y} \in A\subset Y$ is a: 
\begin{itemize}
\item[(i)] positive properly efficient point  of $A$, written $\bar{y}\in $ Pos($A$), if there exists $f \in Y_+^{+i}$ such that $f(\bar{y})\leq f(y)$ for all  $ y \in A$. 
\item[(ii)]  Henig global properly efficient point  of $A$, written $\bar{y}\in $ GHe($A$), if there exists a pointed cone $\mathcal{K}$ satisfying $Y_+\setminus \{0\} \subseteq \mbox{Int} \,\mathcal{K}$ and
$(A-\bar{y})\cap(-\mbox{Int} \,\mathcal{K}) = \emptyset.$
\item[(iii)] Henig properly efficient point  of $A$, written $\bar{y}\in $ He($A$), if for some base $\Theta$ of $Y_+$ there is a scalar $\varepsilon > 0$ such that $\overline{\mbox{cone}}(A-\bar{y})\cap(- \overline{\mbox{cone}}(\Theta+\varepsilon B_{Y})) = \{ 0 \}.$
\item[(iv)]  super efficient point  of $A$, written $\bar{y}\in $ SE($A$), if there is a scalar $\rho > 0$ such that $\overline{\mbox{cone}}(A-\bar{y})\cap (B_{Y}-Y_+) \subset\rho B_{Y}.$
\end{itemize}
\end{definition}
We have the following: Pos($A$) $\subset$ GHe($A$), SE($A$) $\subset$ Min($A$) $\cap$ GHe($A$) $\cap$ He($A$), and 
SE($A$) = He($A$) when $Y_+$ has a bounded base (see, for example, \cite{Ha2010}). 

Next, we define concepts of proper solution and proper minimal point  corresponding to the former notions for programs $(P(0))$ and $(D_{L})$. 
We say that $x_0\in \Omega$ is a positive properly (resp. Hening global properly, Hening properly, super efficient) solution of $(P(0))$ if $f(x_0)\in$ Pos$(\{f(x) \colon  x \in \Omega,\ g(x) \leq 0 \})$ (resp. $f(x_0)\in$ GHe$(\{f(x) \colon  x \in \Omega,\ g(x) \leq 0 \})$, $f(x_0)\in$ He$(\{f(x) \colon  x \in \Omega,\ g(x) \leq 0 \}  )$, $f(x_0)\in$ SE$(\{f(x) \colon  x \in \Omega,\ g(x) \leq 0 \}  )$).  In this case we say  that $f(x_0)$ is a positive (resp. Hening global, Hening, super efficient) minimal point of $(P(0))$. 
Analogously, we say that $x_0\in \Omega$ is a positive properly (resp. Hening global  properly, Hening  properly, super efficient) solution of $(D_{L})$ if $f(x_0)\in$ Pos$(\{f(x)+L(g(x)) \colon  x \in \Omega\}  )$ (resp. $f(x_0)\in$ GHe$(\{f(x)+L(g(x)) \colon  x \in \Omega\}  )$, $f(x_0)\in$ He$(\{f(x)+L(g(x)) \colon  x \in \Omega\}  )$, $f(x_0)\in$ SE$(\{f(x)+L(g(x)) \colon  x \in \Omega\}  )$). 
In this case we say  that $f(x_0)$ is a positive (resp. Hening global, Hening, super efficient) minimal point of $(D_{L})$.

The following theorem claims that each one of the former proper minimal points is, in fact, an ``ordinary'' minimal point with respect to the order induced by some ``particular open'' order cone. This result together with Corollary \ref{CorPrinc} above are the base of the proof of our next Theorem. In the following statement, Min($A\mid \mathcal{K}$) stands the set of minimal points of the set $A \subset Y$ respect to an arbitrary cone $\mathcal{K} \subset Y$ (maybe different to the order cone $Y_+$), i.e., for every  $a \in A$, we have $a \in$ Min($A\mid \mathcal{K}$) if and only if $A\cap (a-\mathcal{K})\subset a + \mathcal{K}$. 

\begin{theorem}(\cite[Theorem 21.7]{Ha2010})\label{thequiv}
Let $Y$ be an ordered normed space having a pointed order cone and $\bar{y} \in A\subset Y$. The following statements hold true:
\begin{itemize}
\item[(i)]  $\bar{y} \in Pos(A )$ if and only if there exists $f \in Y_+^{+i}$ such that defining $Q:=\{ y \in Y : f(y)>0 \}$, we have $\bar{y} \in \mbox{Min}(A\mid Q \cup \{0\})$.
\item[(ii)]  $\bar{y} \in GHe(A )$ if and only if there exists a pointed cone $\mathcal{K}$ such that $Y_+\setminus \{0\}\subseteq \mbox{Int } \mathcal{K}$ and $\bar{y} \in Min(A\mid  \mbox{Int } \mathcal{K} \cup \{0\})$.
\item[(iii)] $\bar{y} \in He(A )$ if and only if there exists a base $\Theta $ of $Y_+ $ such that  for some scalar $ \eta $ satisfying $ 0 < \eta < inf\{ \Vert \theta \Vert : \theta \in \Theta\}$ if we define $\mathcal{K}:=cone(\Theta +\eta \, \mbox{Int} \, B_Y)$, then $\bar{y} \in Min(A\mid \mathcal{K})$.
\item[(iv)]  Assume that $Y_+ $ has a bounded base $\Theta $. Then $\bar{y} \in SE(A )$ if and only if for some scalar $ \eta $ satisfying   $ 0 < \eta < inf\{ \Vert \theta \Vert : \theta \in \Theta\} $ if we define $\mathcal{K}:=cone(\Theta +\eta \, \mbox{Int} \, B_Y)$, then $\bar{y} \in Min(A\mid \mathcal{K})$.
\end{itemize}
\end{theorem}

Now, the result relating our Lagrange multiplier with proper minimals.

\begin{theorem} \label{ThPrEf} 
Fix $y_0\in W(0)$,  let $L_{y_0}$ be the Lagrange process of $(P(0))$ at $y_0$, and assume that $Y_+$ is pointed. The following statements hold true.
\begin{itemize}
\item[(i)]  $y_0$ is a positive (resp. Hening global, Hening)  minimal point of $(P(0))$ if and only if $y_0$ is a positive (resp. Hening global, Hening)  minimal point of $(D_{L_{y_0}})$
\item[(ii)] If $Y_+$ has a bounded base, then $y_0$ is a super efficient minimimal point of $(P(0))$ if and only if $y_0$ is a super efficient minimimal point of $(D_{L_{y_0}})$
\end{itemize}
\end{theorem}

\begin{proof}  
Let us denote $M: =  \{ f(x)+ L_{y_0}(g(x)): x \in \Omega \}$. By Lemma \ref{Lemma_previo_Tma_principal} (ii), we have that $W(0) \subseteq M$. 

The proof of statement (i) falls naturally into three parts, one for each type of proper minimal.
\begin{itemize}
\item[(a)] Case of positive minimal point. Let $y_0$ be a positive minimal point of $(D_{L_{y_0}})$. By Theorem~\ref{thequiv}, $y_0 \in$ Min$(M \mid Q \cup \{0\})$ with $Q=\{y \in Y:f(y)>0\} $, for some $f \in Y_+^{+i}.$ Since $Q$ is open, applying Corollary \ref{CorPrinc}, $y_0 \in$ Min$(W(0) \mid Q \cup \{0\})$. Thus, again by Theorem~\ref{thequiv},  $y_0$ is a positive minimal point of $(P(0))$. Since the results used in our argument above are equivalences, the proof of this part is over.
\item[(b)] Case of Hening global minimal point. Let $y_0$ be a Hening global minimal point of $(D_{L_{y_0}})$. By Theorem \ref{thequiv}, there exists a pointed cone $\mathcal{K}$  with $Y_+\setminus \{0\} \subseteq \mbox{Int} \,\mathcal{K} $ such that $(M-y_0)\cap(-\mbox{Int} \, \mathcal{K}) = \emptyset,$ that is, $y_0 \in$ Min$(M \mid \mbox{Int} \, \mathcal{K} \cup \{ 0 \} ) .$ Applying Corollary \ref{CorPrinc}, $y_0 \in$ Min$(W(0) \mid \mbox{Int} \,\mathcal{K}  \cup \{0\})$. Consequently,   $y_0$ is a Hening global minimal point of $(P(0))$. The results used in our argument are again equivalences.
\item[(c)] Case of Hening minimal point. Let $y_0$ be a Hening  minimal point of $(D_{L_{y_0}})$. By Theorem~\ref{thequiv},  $y_0 \in$ Min$(M \mid Q \cup \{0\})$ with $Q=\mbox{cone}(\Theta +\eta \, \mbox{Int} \, B_Y)$, being $\Theta $ a base of $Y_+ $ and $ \eta $ some scalar satisfying $ 0 < \eta < inf\{ \Vert \theta \Vert : \theta \in \Theta\} .$ Since $Q\setminus \{0\}$ is open, by Corollary~\ref{CorPrinc}, $y_0 \in$ Min$(W(0) \mid Q \cup \{0\})$. Then, by Theorem \ref{thequiv}, $y_0$ is a Hening  minimal point of $(P(0))$.  The results used in our argument are again equivalences.
\end{itemize}
(ii) is a consequence of case (c) above and the equality SE($A $)=He($A $) when $Y_+$ has a bounded base (see \cite{Ha2010}). 
\end{proof}

\section{The Set Optimization Version}\label{Section_Case_Set_Valued}
In this section  we extend the results obtained in the previous sections to the following set-valued optimization problem 
\[
\text{  Min \,} F(x) \text{ \  such that \ } x\in \Omega, \text{ \ } G(x) \cap \left( z-Z_+ \right) \neq \emptyset, \text{ \ \ \ \ \ \ } (\mathcal{P}(z))
\] 
where $\Omega$ is a convex subset of $X$, and  $F:\Omega \rightrightarrows Y$ and $G:\Omega \rightrightarrows Z$ are two set-valued maps such that $F$ is  $Y_+$-convex and $G$ is $Z_+$-convex.
Note that if $G$ is single-valued, the constraint in $(\mathcal{P}(z))$ reduces to $G(x) \in z-Z_+ $ (equiv. $G(x) \leq z$), generalizing the inequalities constraints in $(P(z))$. If, in addition, $F$ is single-valued, then $(\mathcal{P}(z))$ becomes  the conventional convex vector optimization problem $(P(z))$. 
As a consequence, program $(\mathcal{P}(z))$ can be seen as a set-valued generalization of program $(P(z))$. 
For program $(\mathcal{P}(z))$, the graph of the feasible set-valued map in the objective space remains $Y_+$-convex. Hence all the results obtained in the preceding sections regarding to problem $(P(z))$  have their analogous for the set-valued program  $(\mathcal{P}(z))$. 
Their proofs are mimetic adaptations of the previous ones and so we will omit them.  However, we consider worthwhile to show the main definitions and results for the new set-valued program  $(\mathcal{P}(z))$.
In this framework, fixed $z\in Z$ and adapting the notation $\Lambda (z):=\{x \in \Omega\colon  G(x) \cap \left( z-Z_+ \right) \neq \emptyset \}$, $W(z):=\bigcup_{x \in \Lambda (z)}F(x) \subset Y$, we say that a pair $(x_z,y_z)\in \Lambda (z)\times W(z)$ with $y_z \in F(x_z) $  is a minimizer of the program $(\mathcal{P}(z))$ if $y_z\in$ Min$(W(z))$; in this case $y_z$ is said to be a minimal point of $(\mathcal{P}(z))$. Analogously, a pair $(x_z,y_z)\in \Lambda (z)\times W(z)$ with $y_z \in F(x_z) $   is called a  weak (resp. positive properly, Hening global  properly, Hening  properly, super efficient) minimizer of the program $(\mathcal{P}(z))$ if $y_z\in$ WMin$(W(z))$ (resp. Pos$(W(z))$, GHe$(W(z))$,  He$(W(z))$, SE$(W(z))$); in this case $y_z$ is said to be a weak (resp. positive properly, Hening global  properly, Hening  properly, super efficient) minimal point of $(\mathcal{P}(z))$.


Next, we will define the Lagrange multiplier for $(\mathcal{P}(0))$ adapting the procedure followed to get  Definition~\ref{LagrProc}. For every  point $y_0 \in W(0)$ we  define the set $S_{Y_+}(y_0)$ following again equality (\ref{defi_S_Y}). This set enjoys the same properties proved in Section~\ref{The Set-valued Lagrange Process}. Now, for every $h \in S_{Y_+}(y_0)$, we define the corresponding process adapting Definition~\ref{DL} and denoting it by  $\mathcal{L}_h$. 
\begin{definition} \label{LagrProc_Set_Val} 
Fix $y_0 \in W(0)$. We define the Lagrange process of $(\mathcal{P}(0))$ at $y_0$ as
the closed convex process  $\mathcal{L}_{y_0}:Z\rightrightarrows Y $  such that
\begin{equation*}
Graph(\mathcal{L}_{y_0}):=\left\lbrace 
\begin{array}{lll}
\displaystyle \bigcap_{h \in S_{Y_+}(y_0)}{Graph(\mathcal{L}_h)}, & \text{ if } & \mathcal{S}_{Y_+}(y_0) \not= \emptyset, \\
& & \\
\displaystyle Z \times Y, & \text{ if } & \mathcal{S}_{Y_+}(y_0) = \emptyset.
\end{array}
\right. 
\end{equation*}
\end{definition}

\noindent For this new Lagrangian process the following rule of Lagrage multipliers holds. 
\begin{theorem} 
Let $X$, $Y$ and $Z$ be normed spaces such that $Y$ and $Z$ are ordered  by the corresponding cones $Y_+$ and $Z_+$,  both having non empty interior. Take a convex set $\Omega \subset X$, set-valued maps $F:\Omega \rightrightarrows Y$ and $G:\Omega\rightrightarrows Z$ such that $F$ is $Y_+$-convex and  $G$ is $Z_+$-convex.  Assume the existence of a point $x_1 \in \Omega$ for which $-G(x_1)\cap \mbox{Int } Z_+\not = \emptyset$. Then for every $y_0\in F(x_0)\subset  F(\Omega)$ such that $G(x_0) \cap \left( - Z_+ \right) \neq \emptyset$, there exists a closed convex process  $\mathcal{L}_{y_0}:Z\rightrightarrows Y$ such that if $y_0$ is a minimal point of the program
\[
\text{  Min \,} F(x) \text{ \  such that \ } x\in \Omega, \text{ \ } G(x) \cap \left( -Z_+ \right) \neq \emptyset, \text{ \ \ \ \ \ \ } (\mathcal{P}(0))
\]
then $y_0$ is a weak minimal point of the program
\[
\text{  Min } F(x) + \mathcal{L}_{y_0}(G(x)) \text{ such that } x\in \Omega. \hspace{3cm } (\mathcal{D}_{\mathcal{L}_{y_0}})  
\]
In addition, we have the following
\begin{itemize}
\item[(i)]  $y_0$ is a minimal point of $(\mathcal{P}(0))$ if and only if $y_0$ is a minimal point of $(\mathcal{D}_{\mathcal{L}_{y_0}})$, provided  either $Y_+ \setminus \left\lbrace 0 \right\rbrace $ is open or the cone Graph($\mathcal{L}_{y_0}$) has a bounded base.  
\item[(ii)] $y_0$ is a positive (resp. Hening global, Hening) minimal point of $(\mathcal{P}(0))$ if and only if $y_0$ is a positive (resp. Hening global, Hening) minimal point of  $(\mathcal{D}_{\mathcal{L}_{y_0}})$,  provided $Y_+$ is pointed,
\item[(iii)]  $y_0$ is a super efficient point of $(\mathcal{P}(0))$ if and only if $y_0$ is a super efficient minimal point of  $(\mathcal{D}_{\mathcal{L}_{y_0}})$, provided $Y_+$ has a bounded base.
\end{itemize}
\end{theorem}

\section*{Acknowledgement}
We thank the referee for him/her suggestions which have helped us to improve the overall aspect of the manuscript.

\section*{Funding}
The authors have been supported by project MTM2017-86182-P (AEI, Spain and ERDF/FEDER, EU). The author Fernando Garc\'i{}a-Casta\~no has also been supported by MINECO and FEDER (MTM2014-54182).


\begin{thebibliography}{10}
\providecommand{\url}[1]{\normalfont{#1}}
\providecommand{\urlprefix}{Available from: }

\bibitem{GoJa99}
G{\"{o}}tz~A, Jahn~J. {The Lagrange Multiplier Rule in Set-Valued
  Optimization}. SIAM J Optim. 1999;\hspace{0pt}10(2):331--344.

\bibitem{ZheKu06}
Zheng~XY, Ng~KF. {The Lagrange Multiplier Rule for Multifunctions in Banach
  spaces}. SIAM J Optim. 2006;\hspace{0pt}17(4):1154--1175.

\bibitem{Huang2008}
Huang~H. {The Lagrange Multiplier Rule for Super Efficiency in Vector
  Optimization}. J Appl Math Anal Appl. 2008;\hspace{0pt}342(1):503--513.

\bibitem{WanJey00}
Wang~X, Jeyakumar~V. {A Sharp Lagrange Multiplier Rule for Nonsmooth
  Mathematical Programming Problems involving Equality Constraints}. SIAM J
  Optim. 2000;\hspace{0pt}10(4):1136--1148.

\bibitem{BalJim96}
Balb{\'{a}}s~A, {Jim{\'{e}}nez Guerra}~P. {Sensitivity Analysis for Convex
  Multiobjective Programming in Abstract Spaces}. J Math Anal Appl.
  1996;\hspace{0pt}202:645--658.

\bibitem{AmaTaa97}
Amahroq~T, Taa~A. {On Lagrange-Kuhn-Tucker Multipliers for Multiobjective
  Optimization Problems}. Optimization. 1997;\hspace{0pt}41(2):159--172.

\bibitem{White1985}
White~DJ. {Vector Maximisation and Lagrange Multipliers}. Math Program.
  1985;\hspace{0pt}31(2):192--205.

\bibitem{Frenk2007}
Frenk~G, Kassay~G. {Lagrangian Duality and Cone Convexlike Functions}. J Optim
  Theory Appl. 2007;\hspace{0pt}134(2):207--222.

\bibitem{Durea2010}
Durea~M, Dutta~J, Tammer~C. {Lagrange Multipliers for $\epsilon$-Pareto
  Solutions in Vector Optimization with Nonsolid Cones in Banach Spaces}. J
  Optim Theory Appl. 2010;\hspace{0pt}145(1):196--211.

\bibitem{Huang2001}
Huang~Y, Li~Z. {Optimality Condition and Lagrangian Multipliers of Vector
  Optimization with Set-Valued Maps in Linear Spaces}. Oper Res Trans.
  2001;\hspace{0pt}5:63--69.

\bibitem{Huy2012}
Huy~N, Kim~D. {Duality in Vector Optimization via Augmented Lagrangian}. J Math
  Anal Appl. 2012;\hspace{0pt}386(2):473 -- 486.

\bibitem{GarciaMelguizo2015}
Garc{\'i}a~F, Melguizo~Padial~MA. Sensitivity analysis in convex optimization
  through the circatangent derivative. J Optim Theory Appl.
  2015;\hspace{0pt}165(2):420--438.

\bibitem{GarciaCastano2015}
Garc{\'i}a~Casta{\~{n}}o~F, Melguizo~Padial~MA. A natural extension of the
  classical envelope theorem in vector differential programming. J Global
  Optim. 2015;\hspace{0pt}63(4):757--775.

\bibitem{aubin1990set}
Aubin~J, Frankowska~H. Set-valued analysis. Birkh{\"a}user; 1990.

\bibitem{Tanino1980}
Tanino~T, Sawaragi~Y. {Conjugate maps and duality in multiobjective
  optimization}. J Optim Theory Appl. 1980;\hspace{0pt}31(4):473--499.

\bibitem{Hamel2009}
Hamel~AH. {A Duality Theory for Set-Valued Functions I: Fenchel Conjugation
  Theory}. Set-Valued Var Anal. 2009;\hspace{0pt}17(2):153--182.

\bibitem{Hamel2011}
Hamel~AH. {A Fenchel-Rockafellar Duality Theorem for Set-Valued Optimization}.
  Optimization. 2011;\hspace{0pt}60(8-9):1023--1043.

\bibitem{Hamel2014}
Hamel~AH, L{\"o}hne~A. {Lagrange Duality in Set Optimization}. J Optim Theory
  Appl. 2014;\hspace{0pt}161(2):368--397.

\bibitem{Bot2009}
Bot~R, Grad~SM, Wanka~G. {Duality in Vector Optimization}. Springer; 2009.

\bibitem{Zalinescu2015}
Khan~AA, Tammer~C, Zalinescu~C. Set-valued optimization: An introduction with
  applications. Springer-Verlag Berlin Heidelberg; 2015.

\bibitem{Luenberger1969}
Luenberger~DG. Optimization by vector space methods. John Wiley and Sons, Inc.;
  1969.

\bibitem{GARCIACASTANO20151178}
Garc{\'{i}}a-Casta\~no~F, Melguizo~Padial~MA, Montesinos~V. {On Geometry of
  Cones and some Applications}. J Math Anal Appl. 2015;\hspace{0pt}431(2):1178
  -- 1189.

\bibitem{Borwein1993}
Borwein~JM, Zhuang~D. {Super Efficiency in Vector Optimization}. Trans Amer
  Math Soc. 1993;\hspace{0pt}338(1):105.

\bibitem{Casini2010a}
Casini~E, Miglierina~E. {Cones with Bounded and Unbounded Bases and
  Reflexivity}. Nonlinear Anal. 2010;\hspace{0pt}72(5):2356--2366.

\bibitem{Gong1995}
Gong~XH. {Density of the Set of Positive Proper Minimal Points in the Set of
  Minimal Points}. J Optim Theory Appl. 1995;\hspace{0pt}86(3):609--630.

\bibitem{GARCIACASTANO2018}
Garc{\'{i}}a-Casta\~no~F, Melguizo~Padial~MA. {On dentatility and cones with a
  large dual}. To appear in Rev. R. Acad. Cienc. Exactas Fís. Nat. Ser. A Mat. RACSAM. 2019;\hspace{0pt}.

\bibitem{Jahn2004}
Jahn~J. {Vector Optimization. Theory, Applications, and Extensions}. Springer;
  2004.

\bibitem{Ha2010}
Ha~T. Nonlinear analysis and variational problems: In honor of george isac.
  Springer New York; 2010. Chapter Optimality Conditions for Several Types of
  Efficient Solutions of Set-Valued Optimization Problems; p. 305--324.

\end{thebibliography}

\end{document}